\documentclass[10pt]{amsart}
\usepackage{amsmath, amsfonts, amsthm, amssymb}
\usepackage[a4paper, width=16cm, top=30mm, bottom=30mm]{geometry}
\usepackage{mathrsfs}
\usepackage{stmaryrd}
\usepackage{verbatim}
\usepackage{hyperref}
\usepackage{color}
\usepackage{blindtext}
\usepackage{bbm}
\usepackage{ulem}
\usepackage{graphicx}
\usepackage{enumerate}
\newtheorem{theorem}{Theorem}

\DeclareMathOperator*{\argmin}{arg\,min}
\newcommand{\ind}{1\hspace{-2.1mm}{1}}
\newcommand{\RR}{\mathbb{R}}

\newcommand{\NN}{\mathbb{N}}
\newcommand{\EE}{\mathbb{E}}
\newcommand{\Tt}{\mathcal{T}}
\newcommand{\eps}{\varepsilon}
\newcommand{\epsh}{\widehat{\varepsilon}}
\newcommand{\nt}{\widetilde{n}}
\newcommand{\Wh}{\widehat{W}}
\newcommand{\Bh}{\widehat{B}}
\newcommand{\Zh}{\widehat{Z}}

\newcommand{\rrho}{\overline{\rho}}

\begin{document}
\title{How many paths to simulate correlated Brownian motions?}

\author{Antoine Jacquier}
\address{Department of Mathematics, Imperial College London}
\email{a.jacquier@imperial.ac.uk}

\author{Louis Jeannerod}
\address{Ecole Polytechnique Paris}
\email{louis.jeannerod@polytechnique.edu}

\date{\today}

\thanks{The authors are indebted to Piotr Karasinski and Inass El Harrak for suggesting a similar problem
in the context of portfolio simulation, and Johannes Ruf for his comments.}
\maketitle
\begin{abstract}
We provide an explicit formula giving the optimal number of paths needed to simulate two correlated Brownian motions.
\end{abstract}


Monte Carlo methods have a long history, and deep advances have been made over the past decades,
speeding up the computations and reducing the variance of estimators.
We refer the interested reader to the monograph~\cite{Glasserman} for a precise overview of such
techniques, in particular in the context of mathematical finance.
Some particularly powerful advances have been made in multi-level Monte Carlo methods~\cite{Giles}
and functional quantisation~\cite{Pages1, Pages2}.
The latter method aims at approximating a Brownian motion (of infinite-dimensional nature)
by a finite number of paths.
Inspired by this approach, we consider a simpler problem, for which we are able to provide a closed-form solution.

For two correlated standard Brownian motions~$W$ and~$Z$ defined on the same filtered probability space,
we fix a grid $\Tt:=\{t_j = j/m\}_{j=0,\ldots m}$ for fixed $m\in\NN$.
Denoting by~$\Wh^n$ and~$\Zh^n$ the vectors of~$n$ paths, along the grid~$\Tt$, 
approximating~$W$ and~$Z$, 
we answer the following question:

\vspace{0.3cm}
\begin{center}
\fbox{Given a level of error $\eps>0$, 
can we find $\nt<n$ such that the couple $(\Wh^n, \Zh^{\nt})$
is $\eps$-close to $(\Wh^n, \Zh^{n})$?}
\end{center}
\vspace{0.3cm}

The standard simulation approach is to simulate~$n$ paths $(W_n,B_n)$
of a two-dimensional Brownian motion~$(W, B)$,
and deduce the vector $(\Wh^n, \Zh^{n})$ by Cholesky multiplication
since 
$$
\Phi:=
\begin{pmatrix}
W\\Z
\end{pmatrix}
 = 
 \begin{pmatrix}
1 & 0\\ \rho & \rrho
\end{pmatrix}
\begin{pmatrix}
W\\ B
\end{pmatrix},
$$
where $\rrho:=\sqrt{1-\rho^2}$,
and the same identity holds for the approximation vector $\widehat{\Phi}^n := (\Wh^n, \Zh^{n})'$.
Now, if the correlation~$\rho$ is, say, equal to one, then the paths~$B^n$
are superfluous and a waste of computation time.
If $\rho=0$, then the paths~$W^n$ do not provide any information on~$Z$,
and~$n$ paths are needed to simulate~$B$.
For any integer~$n$ and tolerance level~$\eps$, we define a function $\rho \mapsto \nt^n_\eps(\rho)$,
giving the minimal number of paths to approximate~$B$,
and we consider the least square error.
\begin{theorem}
For any $\eps>0$, integer~$n$ and correlation level $\rho \in [-1,1]$, the identity
$$
\nt^n_\eps(\rho)
 = \ind_{\left\lbrace\eps < \frac{n(m+1)(1-\rho^2)}{2m}\right\rbrace}\left\lceil n-\frac{2m\eps}{(m+1)(1-\rho^2)} \right\rceil
$$
holds, where $\lceil\cdot\rceil$ denotes the ceiling function.
\end{theorem}
\begin{proof}
For the optimal vector~$\widetilde{\Phi}^n:=(\Wh^n, \Zh^{\nt})'$, 
we consider an $L^2$-type error of the form
$\epsh := \frac{1}{m}\left\| \widehat{\Phi}^n - \widetilde{\Phi}^n \right\|^2$,
and we can therefore write
$$
\epsh = \frac{1}{m}\sum_{i=1}^{n}\sum_{j=1}^{m} \left(\Zh^{n,i}_{t_j} - \Zh^{\nt,i}_{t_j}\right)^2 
 = \frac{1}{m}\sum_{i=1}^{n}\sum_{j=1}^{m} \left(\Bh^{n,i}_{t_j} - \Bh^{\nt,i}_{t_j}\right)^2 
  = \frac{1-\rho^2}{m}\ind_{\lbrace \nt<n\rbrace} \sum_{i=\nt+1}^{n}\sum_{j=1}^{m} \left(\Bh^{n,i}_{t_j} - \Bh^{\nt,i}_{t_j}\right)^2,
$$
where we write $\Bh^{n} = (\Bh^{n,1}, \ldots, \Bh^{n,n})$.
Confirming the intuition above, we clearly have 
$\nt^1_\eps(\pm 1) = 0$
and
$\nt^1_\eps(0) = 1$.
We start with an initial vector $(\widetilde{\Phi}^{n,1}, \ldots, \widetilde{\Phi}^{n,n})$
consisting of~$n$ two-dimensional discrete paths (each being represented as a vector in~$\RR^{m+1}$).
Since~$B$ is a martingale starting at zero, the guesses~$\Bh^{\nt,i}$ should be null whenever
$i\in\lbrace \nt + 1,\ldots,n\rbrace$,
so that we can write
$$
\epsh = \frac{1-\rho^2}{m}\ind_{\lbrace \nt<n\rbrace} \sum_{i=\nt+1}^{n}\sum_{j=1}^{m} \left(\Bh^{n,i}_{t_j}\right)^2,
$$
and, for any target level~$\eps$, the function $\nt^n_\eps$ then reads
$$
\nt^n_\eps(\rho) = \argmin_{0\leq \nt\leq n}\left\lbrace (n-\nt)\ind_{\lbrace\EE(\hat{\eps}) < \eps\rbrace}\right\rbrace,
$$
with
$$
\EE[\hat{\eps}] = \frac{1-\rho^2}{m} \ind_{\lbrace \nt<n\rbrace} \sum_{i=\nt+1}^{n}\sum_{j=1}^{m} \EE\left[\left(\Bh^{n,i}_{t_j}\right)^2\right]\\
 = \left(1-\rho^2\right) (n-\nt) \sum_{j=1}^{m} \frac{j}{m^2} 
= \left(1-\rho^2\right) (n-\nt)\frac{m+1}{2m},
$$
because $(B_t^2 - t)_{t\geq 0}$ is a martingale, and the $\Bh^{n,i}$ independent.
The target condition $\EE(\hat{\eps}) < \eps$ 
is equivalent to
$\nt > n-\frac{2m\eps}{(m+1)\left(1-\rho^2\right)}$.
If $\frac{2m\eps}{(m+1)\left(1-\rho^2\right)}\geq n$, then clearly $\nt=0$,
otherwise, we rewrite this constraint as
$\eps < \frac{n(m+1)\left(1-\rho^2\right)}{2m}$,
and the theorem follows.
\end{proof}
We illustrate numerically the theorem, plotting the 
map $\rho\mapsto \nt^n_\eps(\rho)$ for different targets~$\eps$,
considering the time interval $[0,1]$.
Not surprisingly, as the intuition suggested, the largest number of paths is attained in the uncorrelated case,
while almost no paths are needed when the Brownian motions are fully (anti)correlated.
The symmetry of the graph is natural, as paths of Brownian motions are symmetric.

\begin{figure}[ht]
\includegraphics[scale=0.5]{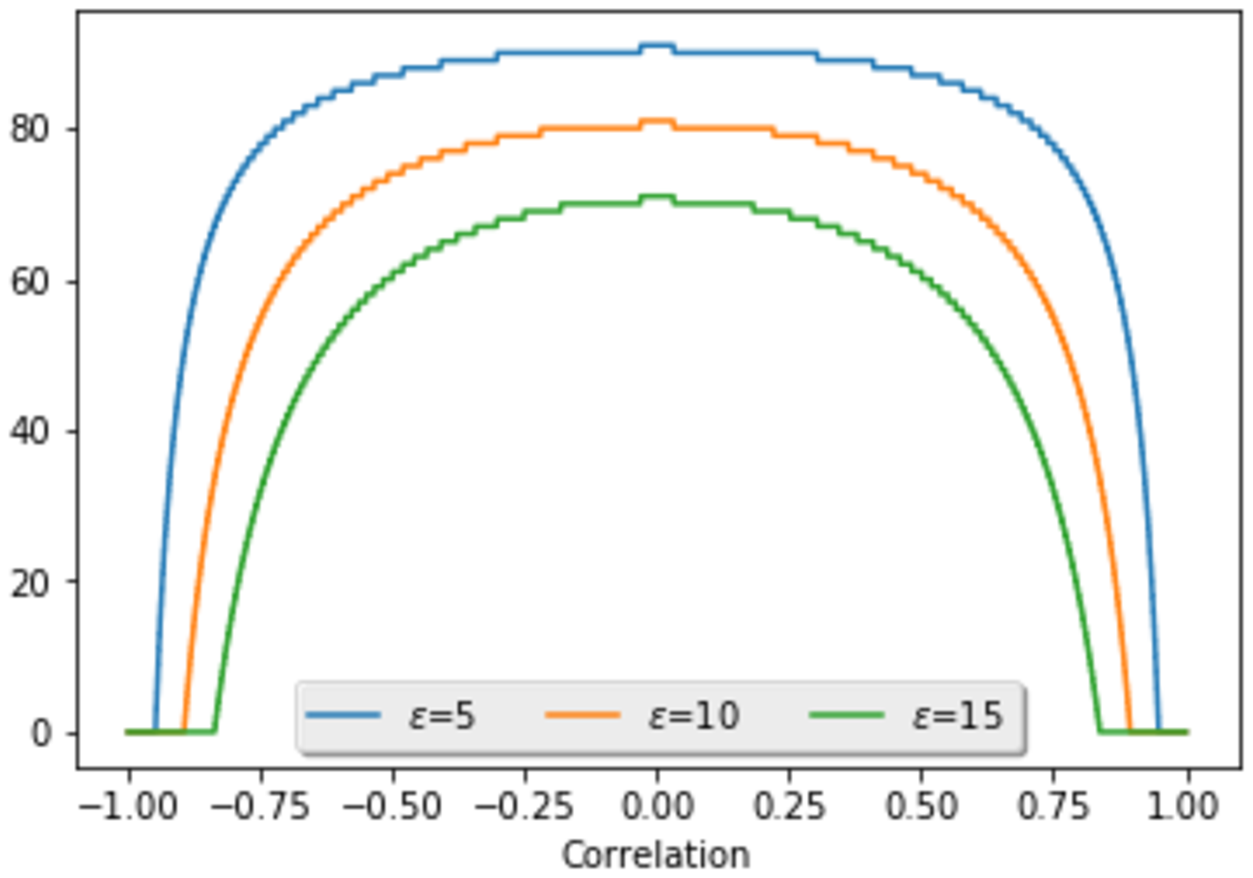} 
\caption{The function $\rho\mapsto\nt^n_\eps(\rho)$ with $n=100$ and $m=100$}
\end{figure}

The result presented here can easily, barring heavier notations, be extended to the multi-dimensional case,
or to more general Gaussian processes.
Furthermore, we only considered here a least square error and, in view of Monte Carlo simulations,
for example when simulating stochastic volatility models in mathematical finance, 
it might be more appropriate to consider other types of errors, such as 
strong $L^2$-error or weak $L^2$-error.
We, however, leave the precise analysis of these schemes to future research.

\end{document}